\documentclass[12pt]{article}
\usepackage{amsmath,amssymb,amsthm}
\usepackage[all]{xy}
\usepackage{graphicx}

\newcommand{\C}{{\mathbb C}}

\newcommand{\N}{{\mathbb N}}

\newcommand{\R}{{\mathbb R}}

\newcommand{\Z}{{\mathbb Z}}

\theoremstyle{plain}
\newtheorem{Theorem}{Theorem}[section]
\newtheorem{Corollary}[Theorem]{Corollary}
\newtheorem{Lemma}[Theorem]{Lemma}
\newtheorem{Proposition}[Theorem]{Proposition}

\theoremstyle{definition}
\newtheorem{Definition}[Theorem]{Definition}
\newtheorem{Remark}[Theorem]{Remark}
\newtheorem{Example}[Theorem]{Example}

\title{Analytic functions in shift-invariant spaces and analytic limits of level dependent subdivision}

\author{
 Maria Charina\\
 Fakult\"at f\"ur Mathematik \\ 
 Universit\"at Wien, Austria
 \and
 Vladimir Yu. Protasov\\
 DISIM of University of L'Aquila, Italy, \\
 and Department of Mechanics and Mathematics \\
 Moscow State University, Russia}

\begin{document}

\maketitle

\begin{abstract}
The structure of exponential subspaces of finitely generated shift-invariant spaces is well understood and the role of such subspaces for the approximation power of refinable function
vectors and related multi-wavelets is well studied.
However, even in the univariate setting, the structure of all analytic subspaces of such shift-invariant spaces has not been yet revealed.
In this paper, in the univariate setting, we characterize all analytic subspaces of finitely generated shift-invariant spaces and provide explicit descriptions of elements of
such subspaces. Consequently, we depict the analytic functions generated by subdivision schemes with  masks of bounded and unbounded support. And we confirm the belief that the exponential polynomials are indeed the only analytic functions generated  by level
dependent (non-stationary) subdivision
schemes with finitely supported masks.
\end{abstract}



\section{Introduction}

\noindent

Shift-invariant spaces have been well studied in the literature in various
contexts. On one hand, such spaces arise  naturally in the signal processing, where they model signals generated by integer shifts of some basic signals~\cite{signals0, signals1, signals2}. On the other hand, approximation properties of shift-invariant spaces have been put to
good use in the study of representation systems, affine synthesis and related issues, see e.g. \cite{FO, Ter}
and references therein.
One of the most popular and well studied applications is the theory of wavelets and multi-wavelets. Deep analysis of structural and approximation properties of
wavelet generated spaces can be found in~\cite{BDR1, BDR2, BDR3, DS, Jia-line, Jia98}.
A closely related subject is subdivision schemes for numerical approximation and for generating curves and
surfaces. The set of limit functions of a subdivision scheme is a shift-invariant
space.

There is a vast number of results on finitely generated shift-invariant spaces.  In this paper, we focus on
the case of compactly supported generators and address the problem of
classifying all analytic functions in the corresponding shift-invariant space. Note that the generators themselves
are not analytic, since they are compactly supported. From the practical point of view
our results can be interesting for both signal processing (for characterizations of
all analytic signals generated by shifts of finitely many signals) and
for wavelet theory (e.g. in the context of wavelet expansions of the solutions
of differential equations). Our original motivation for studying  the structure of analytic subspaces of shift-invariant spaces
comes from  subdivision schemes and refinability.

Intrinsic connection between refinable functions and recursive algorithms called subdivision
schemes is well known \cite{CHM2, CDM, CCGP, ChuiV, Dyn, DynLevinSurvey, ReifPeter08, Protasov_book, W}. Indeed, convergent subdivision schemes generate refinable functions or sequences
of jointly refinable functions. In return, existence of solutions of refinement equations or systems of refinement
equations together with linear independence guarantee the convergence of the underlying subdivision schemes.

Approximation and structural properties of the corresponding shift-invariant spaces depend on the
structure of their polynomial or exponential polynomial subspaces, see e.g.
\cite{CHM1, ContiCotroneiRomani,  DR, ALevin, Ron_intro_s_i_spaces} and references therein.
In most cases the solutions
of refinement equations or the corresponding subdivision limits are not known analytically. Therefore,
the description of their approximation/generation properties is usually done in terms of the finite
sequences of coefficients of the refinement equations or, equivalently, in terms of the corresponding trigonometric polynomials or, equivalently, in terms of the Strang-Fix conditions on the Fourier transforms of
the refinable functions.  Such characterizations are well known in both level independent (stationary)
and level dependent (non-stationary) settings, see e.g. \cite{CHM2, CDM, CharinaContiRomani13, ContiRomani11,
DynLevinSurvey, JePlo}. Moreover, polynomial generation is necessary for linear independence or the convergence
of subdivision in the stationary setting \cite{CDM, CCJZ}. In the non-stationary setting, similar results involving exponential
polynomial generation are not true. Indeed, there exist convergent non-stationary schemes that do not generate any
exponential polynomials \cite{ContiRomaniYoon}. Naturally, a question arises if there exist other
classes of analytic functions rather than exponential polynomials whose generation is necessary
for the convergence of non-stationary subdivision? Furthermore, existence of such classes of functions
would enrich our understanding of the palette of shapes generated by subdivision.

\medskip
 In this paper, we deal with the univariate setting and answer three related \emph{open questions}:

\medskip \noindent
1. What is the exact structure of all analytic subspaces of the shift-invariant spaces
generated by finitely many compactly supported (not necessarily  refinable) functions?

\medskip \noindent 2. How does the number of generators effect the analytic subspaces of the corresponding shift-invariant space?

\medskip \noindent 3. How big are such analytic subspaces  generated by level dependent subdivision schemes with finite masks?

\medskip \noindent
The answer to the first question also provides a characterization of all analytic subspaces generated by vector subdivision schemes. Polynomial subspaces generated by vector subdivision schemes and
approximation properties of the related multi-wavelets were thoroughly studied in the stationary case e.g. in
\cite{CHM1, CH, JiaJiang, MS97, Plonka, PlonkaRon}.

The answer to the second question sheds the light onto  approximation properties of the non-stationary
subdivision scheme with masks of unbounded support e.g. by Rvachev \cite{Rvachev} or of the related constructions in \cite{CohenDyn, Han2}.

\smallskip
This paper is organized as follows: In section \ref{sec:analytic_spaces_phi},  we show first that
the analytic subspace of the shift-invariant space generated by compactly supported (not necessarily refinable) functions $\phi_1, \ldots, \phi_n$ consists of the analytic functions
$$
  f(t)=\sum_{j=1}^s e^{\lambda_j t} \sum_{k=0}^{d_j} \pi_{j,k}(t) \, t^k \, \quad s \in \N,
$$		
for some pairwise distinct modulo $2\pi i$ complex numbers $\lambda_j$, some $d_j \in \N_0$ and some
$1$-periodic analytic  functions $\pi_{j,k}:\R \rightarrow \C$. However this analytic subspace is not merely a
direct sum of spaces of exponential polynomials multiplied by certain $1$-periodic analytic functions.  For
given $\lambda_j$, $d_j \in \N_0$ $\pi_{j,k}:\R \rightarrow \C$, there exist several different shift-invariant
subspaces of analytic functions. The complete characterization of these subspaces is given in Theorem~\ref{th:main_several_generators}.
Indeed, Theorem~\ref{th:main_several_generators} provides an explicit algorithm
for determining all possible analytic subspaces of a given shift-invariant space for given $s$, $d_j$'s.
Moreover, see section \ref{sec:analytic_noref_1}, in the case of a single generator $\phi=\phi_1$,
if the integer shifts of $\phi$ are linearly independent, then the structure of $H$ is characterized by the exponential
decay of the sequences derived from the Fourier transform (analytically extended to $\C$) of $\phi$.
Furthermore, see section \ref{sec:analytic_ref_1}, if additionally the Fourier transform of the generator $\phi=\phi_1$ satisfies the generalized refinability
property
$$
 \hat{\phi}(y)=\prod_{j=1}^\infty \, a_j(2^{-j}y), \quad y \in \R,
$$
with some trigonometric polynomials $a_j$, then the $1$-periodic analytic functions $\pi_{j,k}$ are trigonometric polynomials.
If all trigonometric polynomials $a_j$ are the same, then all $\pi_{j,k}$ are constant. In section \ref{sec:subdivision}, we
first recall the basic facts about subdivision and then interpret the results from section \ref{sec:analytic_spaces_phi}
accordingly. In particular, we confirm the following beliefs: every analytic limit of a stationary subdivision
scheme is a polynomial; every analytic limit of a non-stationary subdivision
scheme is a exponential polynomial.  {\bf Moreover, Theorem~\ref{th:structure_H_n1} together with generalized refinability \eqref{eq:refinable_equation} offer an algorithm for constructing all possible non-stationary
subdivision schemes with desired generation properties.}


\section{Analytic functions in finitely generated shift invariant spaces} \label{sec:analytic_spaces_phi}

\noindent There is a multitude of results in the literature about the properties of shift-invariant spaces
generated by $\Phi=\{\phi_1, \ldots, \phi_n\}$ with each $\phi_j$ of compact support, see \cite{Ron_intro_s_i_spaces} and
references therein. We denote by $V_\Phi$
the corresponding shift-invariant space
$$
 V_\Phi=\operatorname{closure}\left\{ \sum_{j=1}^n \sum_{k \in \Z} c(j,k) \phi_j(\cdot-k)\ : \ c(j,\cdot) \in \ell(\Z) \right\}
$$
with $\ell(\Z)$ being the space of sequences over $\C$. Such spaces for $n=1$ arise in the context of stationary and non-stationary subdivision, while for $n>1$ in the context of vector subdivision schemes and multi-wavelets.

\noindent Our goal is to expose the classes of analytic functions that belong to $V_\Phi$. Under analytic functions \cite{Conway} we mean functions infinitely differentiable on $\R$ and having a power series
expansion around each point in $\R$.


\noindent Our characterizations make use of the following exponential spaces.

\begin{Definition} \label{def:U_exponential} Let $\Lambda \subset \C$ and ${\cal P}_k$ the space of polynomials of degree less than or
equal to $k$. The space $U \subset L_{1,\hbox{loc}}(\R)$ is called \emph{exponential} if
	$$
	 U=\hbox{span} \{ p(\cdot) e^{\lambda \cdot} \ : \ p \in {\cal P}_{k(\lambda)}, \ \lambda \in \Lambda  \}
	 \quad \hbox{dim}(U)=\sum_{\lambda \in \Lambda} (k(\lambda)+1),
	$$
	with the integer $k(\lambda)$ being the multiplicity of $\lambda \in \Lambda$.
\end{Definition}

\noindent In subsection \ref{sec:analytic_noref_n}, we characterize the set of all analytic functions that belong
to $V_\Phi$ without any additional assumptions on the generators in $\Phi$.
In section \ref{sec:analytic_noref_1}, we study the special case of $n=1$ under the assumption that
the integer shifts of $\phi=\phi_1$ are linearly independent. In section \ref{sec:analytic_ref_1}, we additionally
assume the generalized refinability of $\phi=\phi_1$.

\subsection{Case $n \ge 1$} \label{sec:analytic_noref_n}

\noindent In this subsection, we study the structure of the subspace $H$ of analytic functions in
the shift-invariant space $V_\Phi$. The main result of this section, Theorem \ref{th:main_several_generators}, states that
the expected contribution of the exponential spaces $U$ to the structure of $H$ has to be unexpectedly
augmented by contributions of certain $1$-periodic analytic functions. The proof of Theorem \ref{th:main_several_generators}
is based on auxiliary results from subsection \ref{subsec:aux_results} and on
Theorem~\ref{th:structure_H}.

\noindent The next Example shows that the presence of $1$-periodic analytic functions in the subspace $H$ is very natural.

\begin{Example} \label{ex:periodic_functions_in_V} Let $\pi$ be $1$-periodic analytic function. For a B-spline $\psi$, define
$$
 \phi=\psi \cdot \pi.
$$
Then, due to the partition of unity property of $\psi$ and periodicity of $\pi$, we have
$$
 \sum_{k \in \Z} \phi(\cdot-k)=\sum_{k \in \Z} \psi(\cdot-k)\, \pi(\cdot-k)=
  \pi \sum_{k \in \Z} \psi(\cdot-k)=\pi,
$$
i.e. $\pi \in V_\phi$.
\end{Example}

\subsubsection{Auxiliary results} \label{subsec:aux_results}

\noindent To study the structure of the subspace $H$, we make use of the following well known difference operator.

\begin{Definition}  \label{def:exp_diff_operator}
The exponential difference operator $\nabla_\lambda$ on $V_\Phi$ is defined by
$$
 \nabla_\lambda : V_\Phi \rightarrow V_\Phi, \quad  \nabla_\lambda(f)=e^{-\lambda}\, f(\cdot+1)- f, \quad \lambda \in \C.
$$
\end{Definition}

\noindent In Lemma \ref{lem:aux1}, we recall the action of the powers of the difference operator $\nabla_\lambda$ on
the products of the form $\pi(t)p(t) e^{\lambda t}$ with  a $1$-periodic function $\pi$ and an algebraic polynomial $p$.
Such products are shown in the sequel to be the building blocks of the elements of  $H$. Clearly, powers of $\nabla_\lambda$ annihilate  $\pi(t)p(t) e^{\lambda t}$.
On the contrary,  $\nabla_\lambda$ does not affect the structure of such products, if we replace the term $e^{\lambda t}$
by $ e^{\mu t}$ with $\mu \not = \lambda$. We present the proof of this straightforward fact to illustrate that, even in the presence of the $1$-periodic function $\pi$, the action of the difference operator $\nabla_\lambda$ is inherited from its action on the exponential spaces $U$.

\begin{Lemma} \label{lem:aux1} Let $n \in \N$ and $\lambda, \mu \in \C$, $\lambda \not=\mu$. Then for every
$1$-periodic function $\pi: \R \rightarrow \C$ and every polynomial $p \in {\cal P}_{n-1}$, we have
\begin{itemize}
\item[$(i)$] $\nabla_\lambda^n (\pi \, p \,e^{\lambda \, t})=0$,
\item[$(ii)$] $\nabla_\lambda (\pi\, p\, e^{\mu \, t})=\pi\, \widetilde{p}\, e^{\mu \, t}$ for some $\widetilde{p} \in {\cal P}_{n-1}$
with $\deg(\widetilde{p})=\deg(p)$.
\end{itemize}
\end{Lemma}
\begin{proof}
\noindent Part $(i)$: By Definition \ref{def:exp_diff_operator}, we get
$$
 \nabla_\lambda(\pi(t) p(t) e^{\lambda t})=\pi(t) p(t+1) e^{\lambda t}-\pi(t) p(t) e^{\lambda t}=
 \pi(t) e^{\lambda t} \left( p(t+1)-p(t) \right), \quad t \in \R,
$$
where the polynomial $p(\cdot+1)-p \in {\cal P}_{n-2}$, since the leading coefficients of $p(\cdot+1)$ and $p$
are equal. Applying $\nabla_\lambda$ iteratively, we obtain the claim.

\noindent Part $(ii)$: Similarly to $(i)$, the claim follows, due to
$$
  \nabla_\lambda(\pi(t) p(t) e^{\mu t})=\pi(t) p(t+1) e^{\mu (t+1)-\lambda}-\pi(t) p(t) e^{\mu t}=
 \pi(t) e^{\mu t} \left(e^{\mu-\lambda} p(t+1)- p(t) \right),
$$
where the polynomial $\widetilde{p}=e^{\mu-\lambda} p(\cdot+1)-  p \in {\cal P}_{n-1}$, since $\lambda \not = \mu$.
\end{proof}

\noindent Our next result, Lemma \ref{lem:everything_then_one}, states that any $1$-periodic function, appearing in the
representation of $f \in H$, is analytic. It also exposes the finer structure of $H$. We first illustrate the idea of the proof of Lemma \ref{lem:everything_then_one} on the following example.

\begin{Example} Assume that, for $1$-periodic functions $\pi_{1,0}, \pi_{2,0}, \pi_{2,1}:\R \rightarrow \C$ and $\lambda, \mu \in \C$,
$\lambda \not = \mu$, the function
$$
 f(t)=\pi_{1,0}(t) \, e^{\lambda t} +(\pi_{2,0}(t) + \pi_{2,1}(t)\, t) \, e^{\mu t}, \quad t \in \R,
$$
belongs to $H$. Then, applying $\nabla_\lambda$ to eliminate the term with $e^{\lambda t}$, we obtain that
 \begin{eqnarray*}
 \nabla_{\lambda}(f)(t)&=& \pi_{1,0}(t)\, e^{\lambda t}+\left( \pi_{2,0}(t) + \pi_{2,1}(t)\,  (t+1)\right) e^{\mu (t+1)-\lambda}-
 \pi_{1,0}(t)\, \, e^{\lambda t}\\
&-& \left(\pi_{2,0}(t) + \pi_{2,1}(t)\,  t \right) e^{\mu t} =
\Big(\pi_{2,0}(t) \, (e^{\mu-\lambda}-1)+ \pi_{2,1}(t)\,  e^{\mu-\lambda} + \pi_{2,1}(t) \, (e^{\mu-\lambda}-1)\, t \Big) e^{\mu t}
\end{eqnarray*}
belongs to $H$, due to the shift-invariance of $V_\Phi$. Similarly, we apply $\nabla_\mu$ to eliminate
the ''constant'' term in $\nabla_\lambda(f)$ and get that
$$
 \nabla_{\mu} (\nabla_\lambda(f))(t)= e^{-\mu}\nabla_\lambda(f)(t+1)- \nabla_\lambda(f)(t)=
 \pi_{2,1}(t) \,  (e^{\mu-\lambda}-1) e^{\mu t}
$$
is in $H$. This implies also that the $1$-periodic function $\pi_{2,1}$ is analytic, due to the
analyticity of $\nabla_{\mu} (\nabla_\lambda(f))$ and $e^{\mu t}$.  Moreover, the analyticity
of $\nabla_{\lambda}(f)$, $\pi_{2,1}(t) \, e^{\mu t}$ and $\pi_{2,1}(t) \, t\, e^{\mu t}$, implies that
the $1$-periodic function $\pi_{2,0}$ is analytic. And, finally, considering
$$
 f(t)-(\pi_{2,0}(t) + \pi_{2,1}(t)\, t) \, e^{\mu t} =\pi_{1,0}(t) \, e^{\lambda t}
$$
yields that $\pi_{1,0}$ is analytic.
\end{Example}

\noindent Now we are ready to formulate and prove Lemma~\ref{lem:everything_then_one}.

\begin{Lemma} \label{lem:everything_then_one}
Let $H \subset V_\Phi$ be a subspace of all analytic functions. If there exist $s \in \N$, $\lambda_j \in \C$ (pairwise distinct modulo $2\pi i$) and $d_j \in \N_0$, $j=1, \ldots,s$, such that
\begin{equation} \label{eq:f_in_H}
f(t)=\sum_{j=1}^s e^{\lambda_j t} \sum_{k=0}^{d_j} \pi_{j,k}(t) \, t^k \, ,\quad  \pi_{j,k}:\R \rightarrow \C \ \hbox{are $1$-periodic},
\end{equation}
belongs to $H$, then, for $j=1, \ldots,s$,
\begin{description}
\item[$(i)$] $\pi_{j,k}$, $k=0, \ldots, d_j$, are analytic and
\item[$(ii)$] there exist polynomials $p_{j,k}$ with
$\operatorname{deg}(p_{j,k})=k$, $k=0, \ldots, d_j$, such that $\displaystyle e^{\lambda_j t} \sum_{k=0}^{d_j} p_{j,k}(t) \pi_{j,k}(t)$ belong to $H$.	
\end{description}
\end{Lemma}
\begin{proof}
Let $\ell \in \{1, \ldots, s\}$.
Due to Lemma \ref{lem:aux1}, we can eliminate the summands in $f$ with exponential factors $e^{\lambda_j t}$, $j \in \{1,\ldots,s\} \setminus \{\ell\}$,
by applying the corresponding difference operators  $\nabla_{\lambda_j}$ to $f$ each $d_j+1$ times, respectively.
By the shift-invariance of $V_\Phi$ and Definition \ref{def:exp_diff_operator}, the resulting function
$$
  \widetilde{f}(t):=e^{\lambda_\ell t}\, \sum_{k=0}^{d_\ell} \left(\sum_{m=k}^{d_\ell} c_{k,m} \, \pi_{\ell,m}(t) \right) t^k=e^{\lambda_\ell t}\, \sum_{k=0}^{d_\ell} p_{\ell,k}(t) \pi_{\ell,k}(t)  \,  \quad
	 \hbox{belongs to $H$}
$$
and all its coefficients $c_{k,m}$ are non-zero, since $c_{k,m}$ are products of the factors $e^{\lambda_j}$, $j=1,\ldots,d$, or $e^{\lambda_j-\lambda_\ell}-1$,
$j \not = \ell$, and of the binomial coefficients in the expansions of $(t+1)^k$, $k=0, \ldots, d_\ell$.
Therefore, $\operatorname{deg}(p_{\ell,k})=k$, $k=0, \ldots, d_\ell$. Note also that the leading term of $\widetilde{f}$ for $k=d_\ell$ is of the form $c_{d_\ell, d_\ell}\,\pi_{\ell,d_\ell}(t)\, t^{d_\ell} e^{\lambda_\ell t}$. Thus, by Lemma
\ref{lem:aux1}, applying $\nabla_{\lambda_\ell}^{d_\ell}$ to $\widetilde{f}$ leaves us with $d_\ell! \, c_{d_\ell, d_\ell}\,\pi_{\ell,d_\ell}(t)\,
 e^{\lambda_\ell t} \in H$.
The analyticity of $\pi_{\ell,d_\ell}(t)\, e^{\lambda_\ell t}$ and $e^{\lambda_\ell t}$ implies that the function $\pi_{\ell,d_\ell}$ is analytic.
Next we apply to $\widetilde{f}$ the operator $\nabla_{\lambda_\ell}^{d_\ell-1}$ to obtain
$$
    \Big( (d_\ell-1)!c_{d_\ell-1,d_\ell-1} \,\pi_{\ell,d_\ell-1}(t) +  \tilde{c}_{d_\ell-1, d_\ell}\,\pi_{\ell,d_\ell}(t)+ \tilde{c}_{d_\ell, d_\ell}\,\pi_{\ell,d_\ell}(t)\, t\, \Big) e^{\lambda_\ell t} \in H.
$$
Its analyticity and the analyticity of $\pi_{\ell,d_\ell}$ and $e^{\lambda_\ell t}$ imply that $\pi_{\ell,d_\ell-1}$ is analytic. Continuing iteratively yields the claim.
\end{proof}

\subsubsection{Structure of $H$} \label{subsec:structure_H}

\noindent We show first that every $f \in H$ is indeed of the form \eqref{eq:f_in_H} with $1$-periodic analytic
functions $\pi_{j,k}$.

\begin{Theorem} \label{th:structure_H}
Let $H \subset V_\Phi$ be the space of all
analytic functions in $V_\Phi$. Then, there exist $s \in \N$, $\lambda_j \in \C$ (pairwise distinct modulo $2\pi i$),  $d_j \in \N_0$
and $1$-periodic analytic functions $\pi_{j,k}:\R \rightarrow \C$, $k=0,\ldots,d_j$,  $j=1, \ldots,s$,
such that
\begin{equation} \label{def:H}
  H\ \subseteq \ \left\{f \in V_\Phi\,:\, f(t)=\sum_{j=1}^s e^{\lambda_j t} \sum_{k=0}^{d_j} \pi_{j,k}(t)t^k \,   \right\}.
\end{equation}
Moreover,
$$
 \operatorname{dim}(H) \le \sum_{j=1}^n |\operatorname{supp}(\phi_j)|.
$$
\end{Theorem}

\begin{proof}
 Due to the compact supports of $\phi_j$, $j=1, \ldots, n$, there are only finitely many functions
$\phi_j(\cdot-\ell)$, $j=1, \ldots, n$, $\ell \in \Z$, whose supports intersect with the interval $[0,1]$. This finite number we denote by
$$
 N=\hbox{dim}\left(V_\Phi|_{[0,1]} \right) \le \sum_{j=1}^{n} |\hbox{supp}(\phi_j)|.
$$
Therefore, for every $f \in V_\Phi$, the $N+1$ functions $f(\cdot+\ell) \in V_\Phi$, $\ell=0, \ldots, N$, are linearly dependent  over $[0,1]$, namely
$$
 \sum_{\ell=0}^N a_\ell f(t+\ell)=0, \quad t \in [0,1],
$$
where some of the coefficients $a_\ell \in \C$ are non-zero. If, furthermore, $f \in H$, then,
due to the analyticity of this linear combination, we have
\begin{equation} \label{eq:aux2.1}
 \sum_{\ell=0}^N a_\ell f(t+\ell)=0, \quad t \in \R.
\end{equation}
The identity \eqref{eq:aux2.1} implies that, for every $\tau \in [0,1]$, the sequence of
numbers $\{f(\tau+\ell)\ : \ \ell \in \Z\}$ satisfies the linear difference equation with constant
coefficients $a_0, \ldots, a_N$. Let $\alpha_j \in \C$, $j=1, \ldots, s$, $s \le N$, be the roots of the characteristic
polynomial corresponding to \eqref{eq:aux2.1}
and $\mu_j \in \N$ be the multiplicity of $\alpha_j$. Then, the solution of \eqref{eq:aux2.1}, for $\tau \in [0,1]$, has the form
$$
 f(\tau+\ell)=\sum_{j=1}^s e^{\lambda_j \ell}\sum_{k=0}^{d_j} b_{j,k}(\tau) \, \ell^k \, , \quad
d_j:=\mu_j-1, \quad \lambda_j:=\hbox{ln}(\alpha_j), \quad  b_{j,k}: [0,1] \rightarrow \R,
\quad \ell \in \Z.
$$
Extend $b_{j,k}$ to a $1$-periodic function over $\R$. Then substitution $t=\tau+\ell$ and the binomial identity lead to
\begin{eqnarray*}
 f(t)&=& \sum_{j=1}^s e^{\lambda_j(t-\tau)} \sum_{k=0}^{d_j} b_{j,k}(t) \, (t-\tau)^k \,  \\
&=&\sum_{j=1}^s e^{\lambda_j t} \sum_{k=0}^{d_j} \left(\sum_{m=0}^{d_j-k} c_{m,k} \, b_{j,k+m}(t)\, \tau^m \, e^{-\lambda_j \tau} \right) t^k, \quad c_m \in \R, \quad t \in \R.
\end{eqnarray*}
The functions  $\pi_{j,k}: [0,1] \rightarrow \C$, $\displaystyle \pi_{j,k}(\tau)=\sum_{m=0}^{d_j-k} c_{m,k} \, b_{j,k+m}(\tau)\, \tau^m \, e^{-\lambda_j \tau}$, can be extended to $\R$ to be $1$-periodic. By Lemma \ref{lem:everything_then_one} $(i)$,
due to $f \in H$, all the functions $\pi_{j,k}$ are analytic.
\end{proof}

\noindent Finally, in Theorem \ref{th:main_several_generators}, we observe that the intrinsic building blocks of $H$ depend on the invariant spaces of
the shift operators~$A_{d_j}$ with $d_j$, $j=1,\ldots,s$, in \eqref{def:H}. The operator $A_{d_j}$ acts on the space
$$
M_{d_j} \, = \,  \displaystyle \oplus_{k=0}^{d_j} \  {\mathcal P}_{k},
$$
i.e. on the direct sum  of the spaces ${\mathcal P}_{k}$ of algebraic polynomials of degrees at most $k$, by
$$
  A_{d_j}\, (p_0, \ldots, p_{d_j}) \, = \, \bigl(p_0(\cdot + 1), \ldots , p_{d_j}(\cdot + 1)\bigr)
$$
$M_{d_j}$ is a linear space of dimension $\displaystyle \sum_{k=0}^{d_j} (k+1)$ and $A_{d_j}$ is a
linear operator with the block-diagonal matrix representation
\begin{equation}\label{eq.200matr}
A_{d_j}\ = \ \left(
\begin{array}{cccc}
B_{0} & 0 & \cdots & 0\\
0 & B_{1} & \cdots & 0 \\
\vdots & \vdots & \ddots & \vdots \\
0 & \ldots & 0 & B_{d_j}
\end{array}
\right) \  \qquad \hbox{with} \qquad
B_{k}\ = \
\left(
\begin{array}{cccc}
1 & 0 & \cdots & 0\\
k &  1 & \cdots & 0 \\
\vdots & \vdots & \ddots & \vdots \\
1 & \ldots & 1 & 1
\end{array}
\right)
\end{equation}
of size $k+1$,  $k = 0, \ldots ,d_j$ ($B_k$ maps the vector of coefficients of $p_k$ to the coefficients of $p_k(\cdot+1)$).
More precisely, $B_k$ is a lower triangular matrix with ones on the main diagonal and the $\ell$-th column of $B_k$ contains (starting with the main diagonal element) the binomial
coefficients of the expansions of $(t+1)^{k+1-\ell}$, $\ell=1, \ldots,k+1$, respectively.

\begin{Theorem}\label{th:main_several_generators}
Let $H \subseteq V_\Phi$ be the space of all analytic functions
in a finitely generated shift invariant space $V_\Phi$.
Then there exist $s \in \N$, $\lambda_j \in \C$ (pairwise distinct modulo $2\pi i$),  $d_j \in \N_0$,
$1$-periodic analytic functions $\pi_{j,k}:\R \rightarrow \C $, and
subspaces $N_j \subset M_{d_j},$ $j=1, \ldots,s$, each $N_j$ is an invariant subspace of the block-diagonal matrix~$A_{d_j}$ in~(\ref{eq.200matr}) such that
$H$  is a linear span of  spaces~$L_1, \ldots , L_s$ with
\begin{equation}\label{eq.200shift-inv}
L_j \quad = \quad \left\{\,
e^{\lambda_j t}\, \sum_{k=0}^{d_j} \, p_{j, k} \, \pi_{j, k}(t)\ \Bigl| \
\  p_{j, k}(t) \in \mathcal{P}_{k} \ , \
  \bigl(p_{j, 0}, \ldots , p_{j, d_j}  \bigr) \, \in \, N_j \ \right\},
\end{equation}
Moreover, every subspace of $H$ has the same form~(\ref{eq.200shift-inv})
with the same $\lambda_j$ and $\pi_{j, k}$ but with some subspaces
$N_j' \subset N_j, \ j = 1, \ldots , s$.
\end{Theorem}
\begin{proof} For  $f$ in $H$, by Lemma \ref{lem:everything_then_one} $(ii)$, we
get that
$$
g(t)=g_j(t)=e^{\lambda_j t}\, \displaystyle \sum_{k=0}^{d_j} p_{j,k}(t) \pi_{j,k}(t) \in H, \quad j=1,\ldots,s,
$$
with
$ \operatorname{deg}(p_{j,k})=k$, $k=0, \ldots, d_j$.
Note that the operators $e^{\lambda_j-\mu}A_{d_j}-I$, $\mu \in \{\lambda_\ell \ : \ \ell=1,\ldots,s\} \setminus \{\lambda_j\}$, are non degenerate (invertible)
on $M_{d_j}$, $j=1, \ldots,s$ and describe the transformation $f \mapsto g_j$ in Lemma \ref{lem:everything_then_one} $(ii)$.
Moreover, we observe that the shift operator $g(\cdot ) \mapsto g(\cdot +1)$ leaves
the functions $\pi_{j,k}$ unchanged and maps the vector
$\bigl(p_{j,0}, \ldots , p_{j,d_j} \bigr) \in M_{d_j}$ to the vector
$e^{\lambda_j}\, A_{d_j} \, \bigl(p_{j,0}, \ldots , p_{j,d_j} \bigr) \in M_{d_j}$.
Since $H$ is shift-invariant, it contains the
linear span of all integer shifts of $g$. Hence, it contains
all functions $e^{\lambda_j t}\, \displaystyle \sum_{k=0}^{d_j} \tilde p_{j,k} \pi_{j,k}$ with $\bigl(\tilde{p}_{j,0}, \ldots , \tilde{p}_{j,d_j} \bigr)$ from the minimal invariant subspace of the operator
 $A_{d_j}$ that contains the vector $\bigl({p}_{j,0}, \ldots ,{p}_{j,d_j} \bigr)$. The invertibility of $A_{d_j}$ and  $e^{\lambda_j-\mu}A_{d_j}-I$, $\lambda_j \not=\mu$, on $M_{d_j}$
$j=1, \ldots,s$,  completes the proof.
\end{proof}

\begin{Remark}\label{r.200}
{\em Theorem~\ref{th:main_several_generators}  classifies all possible spaces of analytic functions
in finitely generated shift-invariant spaces $V_\Phi$. One takes a finite set
of complex numbers $\lambda_1, \ldots, \lambda_s$ and non-negative integers
$d_j$, $j=1, \ldots,s$. Then, for each $j = 1, \ldots ,s$,
one chooses an arbitrary invariant subspace~$N_j$ of the corresponding block-diagonal matrices $A_{d_j}$. This defines the functional space~(\ref{eq.200shift-inv}).
The direct sum of those $s$ spaces is the space of all analytic functions
in a shift-invariant space $V_\Phi$. It is interesting to note that the matrices
$A_{d_j}$ defined in~(\ref{eq.200matr}) may have a very rich variety of
invariant spaces. It would be interesting to obtain the description of such invariant spaces
at least for $A_{d_j}$ with small number of diagonal-blocks. In the example below we consider the simplest case of two diagonal blocks of sizes 1 and 2 and show that already in this case there are four possible spaces $N_1$}.
\end{Remark}

\noindent The following example shows that the structure of the invariant subspaces of the matrices $A_{d_j}$ in
\eqref{eq.200matr} is highly nontrivial.

\begin{Example} By Theorem~\ref{th:main_several_generators}, every subspace $H$ of analytic functions
generated by the integer shifts of a finite set of compactly supported functions
is a direct sum of spaces $L_j$ of the form \eqref{eq.200shift-inv}.
Consider the simplest case $s=1$, i.e., $H = L_1$, $d_1 = 3$ and the
$3\times 3$ matrix $M_3$ has two blocks, of sizes one and two.
For the sake of simplicity in what follows we denote $\lambda_1 = \lambda,
M_3 = M$ and $N_1 = N$.
Let $\lambda \in \C$, $(a,b,c) \in \R^3$ and
$\pi_{1,0}$, $\pi_{2,2} : \R \rightarrow \C$ be $1$-periodic analytic
functions.
We classify all invariant subspaces $N \subseteq M={\cal P}_0 \oplus {\cal P}_1$ of the linear operator $A: M \rightarrow M$
which, by \eqref{eq.200matr},  has the matrix representation
$$
A= \left(
\begin{array}{cc}
B_0 & 0 \\ 0 &B_1
\end{array}
\right)
\quad \hbox{with} \quad  B_0=1
\quad \hbox{and} \quad B_1= \
\left(
\begin{array}{cc}
1 & 0 \\
1 & 1 \\
\end{array}
\right).
$$
By Theorem~\ref{th:main_several_generators} there is a one-to-one correspondence between these invariant subspaces~$N$ of $A$
and the subspaces $L_1 = L$ in \eqref{eq.200shift-inv}.

\smallskip
\noindent $1).$ The first subspace of the matrix $A$ is
$ N^{(1)} \ = \ \Bigl\{ (0,b, c)  \ \Bigl| \ (b, c) \in \R^2 \Bigr\}$
and the corresponding subspace of the analytic functions is given by
$$
L \ = \ L^{(1)} \ = \ \bigl\{ \, e^{\lambda t} \, \pi_{j,1}(t)\,(bt +c) \ \Bigl| \  (b, c) \in \R^2\, \Bigr\}\, .
$$

\noindent $2).$ The second invariant subspace of the matrix $A$ is
$N^{(2)} \ = \ \Bigl\{ (a, 0, c) \ \Bigl| \  (a, c) \in \R^2 \Bigr\}$ with
$$
L \ = \ L^{(2)}\ = \ \bigl\{ \, e^{\lambda t}  \, \left( \,  a\, \pi_{j,0}(t)  + \, c\,\pi_{j,1}(t)\, \right) \ \Bigl| \
 (a, c) \in \R^2\, \Bigr\}\, .
$$

\noindent $3).$ Moreover, for every vector $(a_0, c_0) \in \R^2\setminus \{0\}$, the matrix $A$ has the following
one-dimensional invariant subspace
$
N^{\, (a_0, c_0)} \ = \ \Bigl\{ \, \tau \, (a_0, 0, c_0)\ \Bigl| \ \tau
 \in \R \Bigr\}\, $ with
$$
L \ = \ L^{\, (a_0, c_0)} \ = \ \Bigl\{ \, e^{\lambda\, t}  \, \left( \,  \tau \, a_0 \, \pi_{j,0}(t)\, + \, \tau \,  c_0 \,
\pi_{j,1}(t)\, \right) \ \Bigl| \
\tau \in \R  \, \Bigr\}\, .
$$
\noindent Surprisingly, these are not all invariant subspaces of~$A$. There is one more family of invariant subspaces.
\smallskip

\noindent $4).$ For every $(u, v) \in \R^2\setminus \{0\}$, the matrix $A$
has a two-dimensional invariant subspace
$$
N^{\, (u, v)} \ = \ \Bigl\{ \,  (a, b, c) \in \R^3\ \Bigl| \
ua + vb = 0\, \Bigr\}\,
$$
with
$$
L \ = \ L^{(u, v)} \ = \ \Bigl\{ \, e^{\lambda t} \, \left(a \,\pi_{j,0}(t) + \pi_{j,1}(t)(bt+c) \right)  \ \Bigl| \
 (a, b, c) \in \R^3, \  ua + vb = 0 \, \Bigr\}\, .
$$
Thus, there are four possible choices for the corresponding space $L_1$.

\noindent Note that even in this simple example the classification of invariant subspaces
of the matrix $A$ is nontrivial.
\end{Example}

\section{Analytic functions in single generated shift-invariant spaces} \label{sec:analytic_noref_1}

\noindent In this section, we characterize the structure of the
subspace $H$ of analytic functions in a singly generated shift-invariant space
$$
 V_{\phi}=\hbox{span} \{\phi(\cdot-\ell) \ : \ \ell \in \Z\}.
$$
This characterization, stated in Theorem \ref{th:structure_H_n1}, relates the structure of $H$ to the
exponential decay of the sequences derived from the Fourier transform  (analytically extended to $\C$)
$$
 \widehat{\phi}(y)=\int_{\R} \phi(x) e^{-2\pi i x \cdot y} dx, \quad y \in \R.
$$
of a compactly supported distribution $\phi$.

\noindent Due to Theorem \ref{th:main_several_generators}, we consider analytic functions
$g(t) = e^{\lambda t} \displaystyle \sum_{k=0}^{d} p_{k}(t)\, \omega_{k}(t)$, $t \in \R$, in $H$ with
polynomials $p_k$ satisfying ${\rm deg}\, p_{k}\, = \, k, \ k = 0, \ldots , d$, $d \in \N_0$.  If $p_j\equiv0$, $j \in \{0,\ldots,d\}$, we set the corresponding $1$-periodic analytic function $\omega_j$ to be identically zero on $\R$. We assume that $\omega_d(t) \ne 0$.

\begin{Theorem} \label{th:structure_H_n1} Let $\phi$ be the generator of the shift-invariant
space $V_\phi$.

\noindent $(i)$ If for $\lambda \in \C$ and $d \in \N_0$, the space $V_{\phi}$
contains $g(t)= e^{\lambda t} \displaystyle \sum_{k=0}^d p_k(t) \omega_k(t)$,
where the functions $\omega_k$ are $1$-periodic analytic, $\omega_d \ne 0$ and the polynomials $p_k$  are of degree $k$, then the sequences
\begin{equation} \label{eq.200decay}
 \left\{ \widehat{\phi}^{(k)}\left(-\frac{i\lambda}{2\pi}+\ell\right) \ : \ \ell \in \Z\right\}, \quad
 k=0,\ldots,d
\end{equation}
decay exponentially as $|\ell|$ goes to infinity.

\noindent $(ii)$ Conversely, if the
sequences in~\eqref{eq.200decay} decay exponentially
for some $\lambda \in \C$ and $d \in \N_0$, then the space $V_{\phi}$ contains the
$(d+1)$-dimensional subspace of analytic functions
\begin{equation}\label{eq.200h}
 H_{\lambda} \ = \ \left\{ \  \sum_{\ell \in \Z} e^{\lambda \, \ell} p(\ell) \, \phi (\cdot-\ell)\ : \
 p \in {\mathcal P}_d\ \right\}
\end{equation}
spanned by
\begin{equation}\label{eq.200basis}
\sum_{\ell \in \Z}   e^{\lambda \, \ell}\, \ell^k \, \phi(t-\ell) =  e^{\lambda t}
\sum_{j=0}^k \, {k \choose j} \,  t^{k-j} \, (-1)^{j} \,  \omega_j(t), \quad t \in \R, \quad  k = 0, \ldots, d,
\end{equation}
where $\omega_k$ are $1$-periodic analytic functions given by
\begin{equation}\label{eq.200pis}
\omega_k(t)\quad= - \left(\frac{-1}{2\pi i} \right)^k \, \sum_{\ell \in \Z} \widehat{\phi}^{\, (k)}
\left( - \frac{i\, \lambda}{2\pi} + \ell\right)\,  e^{2\pi i \ell t}\quad
 k = 0, \ldots, d, \quad t \in \R.
\end{equation}
\end{Theorem}
\begin{proof} Let $\psi(t)=e^{-\lambda t} \phi(t)$, $t \in \R$.

\noindent We first prove $(ii)$.
Assume that the sequences in~\eqref{eq.200decay} decay exponentially. Then, by Payley-Wiener theorem, the $1$-periodic functions in~\eqref{eq.200pis}
are analytic and, by the Poisson summation formula, we have
$$
\omega_k(t)= - \left(\frac{-1}{2\pi i} \right)^k \sum_{\ell \in \Z} \widehat{\phi}^{\, (k)}
\left( - \frac{i\, \lambda}{2\pi} + \ell\right)\,  e^{2\pi i \ell t}= - \left(\frac{-1}{2\pi i} \right)^k\,
\sum_{\ell \in \Z} \, \widehat{\psi}^{(k)}(\ell) \,  e^{2\pi i \ell t}=
\sum_{\ell \in \Z} (t-\ell)^k \, \psi(t-\ell)
$$
for $k = 0, \ldots , d$ and $t \in \R$.
Hence, for $k=0, \ldots,d$, the functions
\begin{eqnarray*}
\sum_{\ell \in \Z} \ell^k\, \psi(t-\ell) &=&
 \sum_{\ell\in \Z} \Bigl(\, t - (t-\ell)\, \Bigr)^k\, \psi(t-\ell) \quad = \quad
\sum_{\ell\in \Z} \sum_{j=0}^k  {k \choose j} \, t^{k-j} \, (-1)^{j} \, (t-\ell)^{j}\, \psi(t-\ell) \\
&=&   \sum_{j=0}^k \,  {k \choose j } \, t^{k-j} \, (-1)^{j} \,
\sum_{\ell\in \Z}  \, (t-\ell)^{j}\, \psi(t-\ell)\ = \
\sum_{j=0}^k \, {k \choose j} \, t^{k-j} \, (-1)^{j} \,   \omega_j(t), \quad t \in \R,
\end{eqnarray*}
are analytic and belong to $H_\lambda \subseteq V_\phi$.

\noindent  The proof of $(i)$ is by induction on $d$. In the case $d=0$, the polynomial $p_0$ is constant, w.l.g $p_0(t)\equiv 1$. Then $g(t)=e^{\lambda t} \omega(t) \in V_\phi$ if and only if $\omega \in V_\psi$, i.e.
\begin{equation} \label{aux1}
 \omega(t)=\sum_{\ell \in \Z} a_\ell \, \psi(t-\ell), \quad a_\ell, t \in \R.
\end{equation}
The periodicity of $\omega$ implies that
$$
  \sum_{\ell \in \Z} a_\ell \, \psi(t+1-\ell)=\sum_{\ell \in \Z} a_\ell \, \psi(t-\ell), \quad t \in \R,
$$
which is equivalent to the identity
$$
 \sum_{\ell \in \Z} (a_{\ell+1}-a_\ell) \, \psi(t-\ell)=0, \quad t \in \R.
$$
Due to the linear independence of the integer shifts of $\psi$, we obtain $a_{\ell+1}-a_\ell=0$ for all $\ell \in \Z$.  Or, equivalently, w.l.g. $a_\ell=1$ for all $\ell \in \Z$. Therefore, by the Poisson summation formula, we obtain
$$
 \omega(t)=\sum_{\ell \in \Z} \psi(t-\ell)=-\sum_{\ell \in \Z} \widehat{\psi}(\ell) \, e^{2\pi i \ell t}=-\sum_{\ell \in \Z} \widehat{\phi}\left(-\frac{i\lambda}{2\pi}+\ell\right) e^{2\pi i \ell t}, \quad
t \in \R.
$$
Due to the analyticity of $\omega$, the above identity holds if and only if there exists a constant $C>0$ and $q \in (0,1)$ such that
$$
  \left| \, \hat{\phi} \left(-\frac{i\lambda}{2\pi}+\ell\right) \, \right| \le \,C \,q^{|\ell|} \quad \hbox{for all} \quad \ell \in \Z.
$$

\noindent We assume that the hypothesis is true for $d-1$. Then
$g= \displaystyle \sum_{k=0}^{d} p_k\, \omega_k$, $\omega_{d} \ne 0$ belongs to $V_\psi$ if and only if
$g = \displaystyle \sum_{\ell \in \Z} a_\ell \psi(\cdot-\ell)$. Thus, by  the periodicity of
$\omega_k$, $k=0, \ldots,d$, we get
$$
g(t+1)-g(t)=\sum_{\ell \in \Z} \left( a_{\ell+1} - a_\ell\right) \, \psi (t - \ell) \ = \
\sum_{k=0}^{d} \Big( p_k(t-1) - p_k(t) \Big)  \, \omega_k(t), \quad t\in \R.
$$
Define $q_k:=p_k(\cdot-1) - p_k$, $k=1, \ldots,d$. Due to  $p_0(t)\equiv 1$ and $\operatorname{deg}(q_k)=k-1$ for
$k=1, \ldots,d$, the function
$$
\tilde{g}(t):=\sum_{\ell \in \Z} \left( a_{\ell+1} - a_\ell\right) \, \psi (t - \ell) \ = \
\sum_{k=0}^{d-1} q_{k+1}(t) \, \omega_{k+1}(t), \quad t\in \R,
$$
satisfies the inductive assumption. Therefore, the sequences in \eqref{eq.200decay}
for $k=0, \ldots,d-1$, decay exponentially.
Secondly, by $(ii)$, the structure of the $d$-dimensional $H_\lambda$ and the analyticity of $\tilde{g}$,
imply that
$$
 \sum_{\ell \in \Z} \bigl( a_{\ell+1}  -  a_\ell\bigr) \, \psi (\cdot - \ell)\, = \,
 \sum_{\ell \in \Z} p(\ell) \, \psi (\cdot - \ell), \quad \hbox{for some} \quad p \in {\cal P}_{d-1}.
$$
The linear independence of the integer shifts of $\psi$ implies that $a_{\ell+1}-a_\ell= p(\ell)$
for all $\ell \in \Z$. Theory of difference equations ensures that every solution of this
difference equation is given by $a_\ell = \tilde p(\ell)$, $\ell \in \Z$, for some polynomial $\tilde p$
of degree $d$. We write $\tilde p(\ell) = \alpha \, \ell^{d} + q(\ell)$, $\alpha \in \R \setminus \{0\}$ and  ${\rm deg } \, (q )\le d-1$, and have
$$
g(t) \ = \ \sum_{\ell \in \Z} \Big(\alpha \ell^{d} + q(\ell)\Big) \, \psi(t-\ell) \ = \
 \sum_{k=0}^{d} p_k(t) \, \omega_k(t), \quad t \in \R,
$$
and, hence, the function
\begin{equation} \label{eq:representation_of_poly}
\alpha \sum_{\ell \in \Z} \ell^{d}\, \psi(t-\ell) \quad  =  \quad
- \, \sum_{\ell \in \Z} q(\ell)\, \psi(t-\ell)    \ + \
 \sum_{k=0}^{d} p_k(t) \, \omega_k(t), \quad t \in \R,
\end{equation}
is analytic due to the analyticity of $g$ and, by the inductive assumption, analyticity of
$\displaystyle \sum_{\ell \in \Z} q(\ell)\, \psi(\cdot-\ell)$. Consequently and due to the
inductive assumption,
the function
$$
 \sum_{\ell \in \Z} (t-\ell)^{d} \, \psi(t-\ell)= (-1)^d \, \sum_{\ell \in \Z} \ell^d \, \psi(t-\ell)+
\sum_{k=1}^d \, (-1)^{d-k}\,t^k \, \sum_{\ell \in \Z} \, \ell^{d-k} \, \psi(t-\ell), \quad t \in \R,
$$
is analytic as well.  Therefore, by the Poisson summation formula
$$
 \sum_{\ell \in \Z} \, (t-\ell)^{d} \, \psi(t-\ell)= - \left(\frac{-1}{2\pi i} \right)^d \, \sum_{\ell \in \Z} \widehat{\psi}^{(d)}(\ell) \, e^{2\pi i \ell t}= - \left(\frac{-1}{2\pi i} \right)^d \, \sum_{\ell \in \Z} \widehat{\phi}^{(d)}
\left( -\frac{i\lambda}{2\pi}+\ell\right) \, e^{2\pi i \ell t}, \quad t \in \R,
$$
and the analyticity implies that the sequence $ \left\{ \widehat{\phi}^{(d)}\left(-\frac{i\lambda}{2\pi}+\ell\right) \ : \ \ell \in \Z\right\}$ decays exponentially. Thus, we have shown that \eqref{eq.200decay} is satisfied for $k=0, \ldots, d$.
\end{proof}

\begin{Corollary}\label{c.200}
The set of analytic functions spanned by the shifts of a compactly
supported function $\phi$ is a linear span of spaces $H_{\lambda}$ in  \eqref{eq.200h}
over all $\lambda \in \C$ such that the sequences in~\eqref{eq.200decay} decay exponentially.
\end{Corollary}

\section{Single generated shift-invariant spaces with generalized refinability} \label{sec:analytic_ref_1}

\noindent Additional assumption on the generalized refinability of $\phi$, i.e. the property
\begin{equation} \label{eq:refinable_equation}
 \widehat{\phi}(y)=\prod_{j=1}^\infty \, a_j(2^{-j}y),
\end{equation}
for some trigonometric polynomials
$$
 a_j(y)=\sum_{m \in Z} {\rm a}_{j,m} \, e^{-2\, \pi \, i \, m \, y}, \quad {\rm a}_{j,m} \in \R \quad y \in \R,
$$
replaces the requirement in Theorem \ref{th:structure_H_n1} on the exponential decay of sequences in \eqref{eq.200decay}
by a requirement
that only finitely many of the sequence elements are non-zero (i.e. the corresponding $1$-periodic analytic functions
in \eqref{eq.200pis}  are trigonometric polynomials).

\noindent The main result of this section finalizes our knowledge about the structure of $H$.

\begin{Theorem} \label{th:main_refinable}
Let $V_\phi=\operatorname{span}\{\phi(\cdot-\ell)\ : \ \ell \in Z \}$ be defined by
$ \displaystyle \widehat{\phi}=\prod_{j=1}^\infty \, a_j(2^{-j} \cdot)$
with the trigonometric polynomials $a_j$ satisfying
$$
 \operatorname{deg}(a_j) \le N, \quad a_j(0)=1 \quad \hbox{and} \quad \|a_j\|_\infty \le C < \infty, \quad j \in \N.
$$
If, for some $\lambda \in \C$ and $d \in \N_0$, $d \le N$, the analytic function $e^{\lambda \, t}\displaystyle \sum_{k=0}^{d}\, p_k\, \omega_k$, $\omega_{d} \ne 0$  belongs to $V_\phi$, then
the sequences
\begin{equation} \label{eq.200decay2}
 \left\{ \widehat{\phi}^{(k)}\left(-\frac{i\lambda}{2\pi}+\ell\right) \ : \ \ell \in \Z\right\}, \quad
 k=0,\ldots,d
\end{equation}
contain (all together) at most $N$ non-zero elements.
\end{Theorem}

\noindent The inductive proof of Theorem~\ref{th:main_refinable} follows from Propositions~\ref{prop:d=0} and \ref{prop:d} and
Lemma~\ref{lem:estimate_4_a(y)}. Proposition~\ref{prop:d=0} provides the base of the inductive proof of Theorem~\ref{th:main_refinable} in the case $d=0$. The inductive step in Proposition~\ref{prop:d} is proven similarly to Proposition~\ref{prop:d=0},
yet there are crucial differences that we point out. Both  Propositions~\ref{prop:d=0} and \ref{prop:d} rely on the result of
Lemma~\ref{lem:estimate_4_a(y)}.

\noindent
In the proof of Proposition~\ref{prop:d=0} we use the idea of the method of
counting of zeros elaborated in~\cite{Protasov2}. The essence of the
method is the following: if the infinite product of trigonometric polynomials has too many zeros on a segment $[0, r]$, then one of the polynomials
must have more than $N$ zeros which leads to the contradiction.
However, for proving Proposition~\ref{prop:d=0}, this idea should be
significantly modified since here we have to count not zeros but in a sense
``almost zeros'' of polynomials.  That is why we begin with
 Lemma \ref{lem:estimate_4_a(y)}, which states that the maximum norm
of a trigonometric polynomial of degree $N$ is small, if its point evaluations at arbitrary (well-separated) $N+1$ pairwise distinct points in $[0,1)$ are small. This result generalizes the well known fact that an algebraic
polynomial of degree $N$ is identically zero, if it vanishes at $N+1$ points.

\begin{Lemma} \label{lem:estimate_4_a(y)}
Let $ \displaystyle
 a(y)=\sum_{m=0}^N {\rm a}_m \, e^{-i\, 2 \, \pi \,m \, y}$, $y \in \R$, $N \in \N$ and $y_m \in [0,1)$, $m=0,\ldots,N$ be pairwise distinct. Then
\begin{equation}\label{eq.roots-est}
 \|a\|_\infty \left( \min_{m,k=0,\ldots, N \atop m \not =k} |y_m-y_k|\right)^N
 \quad \le \quad 2^{-N} (N+1) \max_{m=0,\ldots,N} |a(y_m)|
\end{equation}
\end{Lemma}
\begin{proof}
By the Lagrange interpolation formula, the trigonometric polynomial $a$ satisfies
$$
 a(y)=\sum_{m=0}^N a(y_m) \, \prod_{k=0 \atop k \not=m} \frac{z-z_k}{z_m-z_k}, \quad z=e^{-i2\pi y}, \quad
 z_m=e^{-i2\pi y_m}, \quad m=0, \ldots, N.
$$
Thus, due to $ \displaystyle \prod_{k=0 \atop k \not=m} |z_m-z_k| \ge \left( \min_{m,k=0,\ldots, N \atop m \not =k}
|z_m-z_k|\right)^N$, on the unit circle
we have
$$
 |a(y)|  \left( \min_{m,k=0,\ldots, N \atop m \not =k} |z_m-z_k|\right)^N \le \max_{m=0,\ldots,N} |a(y_m)|
 \sum_{m=0}^N \, \prod_{k=0 \atop k \not=m} |z-z_k| \le 2^N (N+1) \max_{m=0,\ldots,N} |a(y_m)|,
$$
where the chord length $|z-z_k| \le 2$ for $|z|=1$.
On the other hand, the length of an arbitrary  chord of a unit circle
is at least the length of the shortest arc defined by this chord multiplied by
$\frac{2}{\pi}$ (this estimate is achieved for diameters).
Therefore, $|z_m-z_k| \, \ge \, \frac{2}{\pi}\, \cdot \, 2\pi \, \cdot\, |y_m-y_k|\, = \, 4\, |y_m-y_k|$.  Thus,
$$
|a(y)|\,   \left( \min_{m,k=0,\ldots, N \atop m \not =k} |z_m-z_k|\right)^N \ \ge \
|a(y)| \, 4^{N} \, \left( \min_{m,k=0,\ldots, N \atop m \not =k} |y_m-y_k|\right)^N \, .
$$
Consequently, for every $y \in [0,1)$, we have
$$
|a(y)| \,  \left( \min_{m,k=0,\ldots, N \atop m \not =k} |y_m-y_k|\right)^N
\ \le \  2^{-N} (N+1) \max_{m=0,\ldots,N} |a(y_m)|\,.
$$
Taking maximum over $y \in [0,1)$, we arrive at the desired estimate \eqref{eq.roots-est}.
\end{proof}

\noindent Now we are ready to prove Theorem \ref{th:main_refinable}.

\begin{Proposition} \label{prop:d=0}
The statement of Theorem \ref{th:main_refinable} holds for $d=0$.
\end{Proposition}
\begin{proof}
Let $d=0$  and assume that there are at least $N+1$
non-zero elements in the corresponding sequence in \eqref{eq.200decay} with $\alpha=-\frac{i\lambda}{2\pi}$, $\lambda \in \C$.

\noindent \emph{1.Step:} W.l.g. $\widehat{\phi}(\alpha) \not=0$.
Then by \eqref{eq:refinable_equation}, for every $\varepsilon>0$ there exists $r_\alpha \in \N$ such that for all $r \ge r_\alpha$
\begin{equation} \label{aux101}
 \left|\prod_{j=1}^{r} a_j(2^{-j}\alpha) - \widehat{\phi}(\alpha)\right|<\varepsilon,
\end{equation}
thus, this product is bounded away from zero uniformly for all $r \ge r_\alpha$. The fact that $a_j(0)=1$, $j \in \N$, implies that
there exists $R \in \N$ such that for all $r \ge r_\alpha$
\begin{equation} \label{aux102}
 \left|\prod_{j=r+R+1}^\infty a_j(2^{-j}\alpha+2^{-j+r})-1\right|<\varepsilon,
\end{equation}
i.e. the above product is also uniformly bounded away from zero.
Next, we split the infinite product appearing in the definition of $\widehat{\phi}(\alpha+\ell)$ into three products accordingly to
the properties in \eqref{aux101}-\eqref{aux102}. For $\ell=2^r$, $r \ge r_\alpha$, due to the 1-periodicity of the trigonometric polynomials $a_j$, we have
$$
 \widehat{\phi}(\alpha+2^r)=\prod_{j=1}^r a_j(2^{-j}\alpha) \, \prod_{j=r+1}^{r+R} a_j(2^{-j}\alpha+2^{-j+r}) \, \prod_{j=r+R+1}^\infty
 a_j(2^{-j}\alpha+2^{-j+r}).
$$
Due to \eqref{aux101}-\eqref{aux102}, the exponential decay of the sequence in \eqref{eq.200decay} implies that
\begin{equation}\label{aux102_a}
 \left|\prod_{j=r+1}^{r+R} a_j(2^{-j}\alpha+2^{-j+r}) \right| \le C q^{2^r}, \quad q \in (0,1).
\end{equation}
Hence, at least one of the factors in \eqref{aux102_a} (\emph{has an almost zero}) is in the absolute value smaller than
or equal to $Cq^{2^r/R}$.
Repeating the argument with $\ell=2^{r+n}$, $n \in \N$, we conclude that $J$ trigonometric polynomials
$a_{r+1}, \ldots a_{r+1+J}$, $r \ge r_\alpha$, $J >>R$, have at least $J-R$  almost zeros. The possible almost zeros for
each $a_k$, $k \in \{r+1,\ldots,r+1+J\}$ are at the distinct complex points in \eqref{aux102_a}
\begin{equation}\label{aux103}
 2^{-k}\alpha+2^{-1}, 2^{-k}\alpha +2^{-2}, \ldots, 2^{-k}\alpha +2^{-R}.
\end{equation}

\noindent \emph{2.Step:} By assumption, there exist other $N$ distinct $\alpha_\ell=\alpha+\ell$, $\ell \in L \subset \N$, such that
$\hat{\phi}(\alpha_\ell) \not=0$. We repeat the argument in \emph{1.Step} with $\alpha=\alpha_\ell$ for these $N$ distinct
$\alpha_\ell$ and conclude that $J$ trigonometric polynomials
$a_{r+1}, \ldots a_{r+1+J}$, $r\ge \max\{r_\alpha, r_{\alpha_\ell}\}$, $J >>R$,  have (together with the almost zeros from \emph{1.Step}) at
least $(N+1)(J-R)$ almost zeros. The possible almost zeros for each $a_k$, $k \in \{r+1,\ldots,r+1+J\}$ are at the distinct complex points
\begin{equation}\label{aux104}
 2^{-k} \alpha_\ell+2^{-1}, 2^{-k} \alpha_\ell +2^{-2}, \ldots, 2^{-k} \alpha_\ell +2^{-R}.
\end{equation}
Thus, due to $J>>R$, on average there are at least
$$
 (N+1)\left(1-\frac{R}{J} \right)>N
$$
almost zeros for each $a_{r+1}, \ldots a_{r+1+J}$ and, by the pigeonhole principle, there exists $a_k$, $k \in \{r+1,\ldots,r+1+J\}$  with
$N+1$ almost zeros of the form in \eqref{aux103}-\eqref{aux104}.

\noindent \emph{3.Step:} We use Lemma~\ref{lem:estimate_4_a(y)} to get a contradiction to the fact that $a_k(0)=1$.
First note that all the points (we set $\alpha_0=\alpha$)
$$
  w_\ell=2^{-k}\alpha_\ell+2^{-s}, \quad s \in \{1, \ldots,R\}, \quad \ell \in L \cup\{0\}, \quad |L|=N+1.
$$
have, due to $\alpha=-\frac{i\, \lambda}{2\pi}$, $\lambda \in\C$, the same imaginary part
$$
 \operatorname{Im}(w_\ell)=-2^{-k} (2 \pi)^{-1} \, \operatorname{Re}(\lambda), \quad \ell \in L \cup \{0\}.
$$
Moreover, these points $w_\ell$ are separated by the distance of at least $2^{-k}$ for $k > R$. Indeed, let $n,s \in \{1, \ldots, R\}$, $n\not=s$,
and $\ell,\tilde{\ell} \in L \cup\{0\}$, $\ell \not=\tilde{\ell}$. Then, for $k>R$, due
to $|\alpha_\ell-\alpha_{\tilde{\ell}}|=|\ell-\tilde{\ell}|$  being an integer bigger than or equal to $1$, we have
$$
 |2^{-k}(\alpha_\ell-\alpha_{\tilde{\ell}})+2^s-2^n| \ge |2^s-2^n|-2^{-k}|\alpha_\ell-\alpha_{\tilde{\ell}}| \ge 2^{-R}-2^{-k}\ge 2^{-k+1}-2^{-k}=2^{-k}.
$$
Secondly, for
$$
  a_k(y)=\sum_{m=0}^N {\rm a}_{k,m}\, e^{-i2 \, \pi \, m \, y}, \quad y \in \R,
$$
define the polynomial
$$
 \tilde{a}_k(y)=\sum_{m=0}^N \tilde{{\rm a}}_{k,m} e^{-i 2\pi m y}, \quad \tilde{{\rm a}}_{k,m}={\rm a}_{k,m}\, e^{-2^{-k}\,\operatorname{Re}(\lambda)\,m}, \quad m=0, \ldots,N.
$$
Note that $\displaystyle \lim_{k \rightarrow \infty}\|a_k-\tilde{a}_k\|_\infty=0$ and
that the minimal distance between the real points (which are real parts of $w_\ell$'s)
$$
 y_\ell=2^{-k}\left( \frac{\operatorname{Im}(\lambda)}{2\pi} +\ell \right) +2^{-s}, \quad \ell \in L \cup \{0\},
$$
is given by $2^{-k}$, due to all $w_\ell$'s having the same imaginary part. Also note that
the almost zeros $w_\ell$ of $a_k$ are closely related to the almost zeros of $\tilde{a}_k$ by $a_k(w_\ell)=\tilde{a}_k(y_\ell)$,
$\ell \in L \cup \{0\}$. Therefore, by  Lemma~\ref{lem:estimate_4_a(y)}, we get
$$
 \|\tilde{a}_k\|_\infty \, 2^{-k} \le 2^{-N}(N+1) \, C \, q^{2^{-k}/R}.
$$
On the other hand, $a_k(0)=1$ and  $\displaystyle \lim_{k \rightarrow \infty}\|a_k-\tilde{a}_k\|_\infty=0$ lead to a contradiction.
\end{proof}

\noindent Next we provide the inductive step that completes the proof of  Theorem \ref{th:main_refinable}.

\begin{Proposition} \label{prop:d}
The statement of Theorem \ref{th:main_refinable} holds for $d \in \N$.
\end{Proposition}
\begin{proof}
The base of the induction follows from Proposition~\ref{prop:d=0}. We assume that, for $k=0,\ldots,d-1$, the
sequences in \eqref{eq.200decay} with $\alpha=-\frac{i\lambda}{2\pi}$, $\lambda \in \C$, have in total finitely
many non-zero elements. This implies the existence of $r_0$ such that for all $r>r_0$ we hate
$$
 \hat{\phi}^{k}(\alpha+2^r)=0, \quad k=0, \ldots,d-1.
$$
The inductive step we prove by contradiction assuming that there are at least $N+1$
non zero elements in the sequences in \eqref{eq.200decay} for $k=0,\ldots,d$.

\noindent \emph{1.Step} By the argument in Proposition~\ref{prop:d=0} \emph{1.Step}, $\hat{\phi}(\alpha) \not =0$ and, for arbitrary $r \ge r_\alpha$, the trigonometric polynomials $a_{r+1}, \ldots a_{r+1+J}$, $r \ge r_\alpha$, $J>>R$, have at least $J-R$  almost zeros of the form in \eqref{aux103}.

\noindent \emph{2.Step} For the same $\alpha$, the additional information about the exponential decay of the other sequences in \eqref{eq.200decay} for $k=1, \ldots,d$  supplies another $d(J-R)$ almost zeros. Indeed, due to $\hat{\phi}(\alpha) \not =0$, there exist $\rho \in (0,1)$ and a constant $C_0>0$  such that
\begin{equation} \label{aux400}
|\hat{\phi}(\alpha+t)|\ge C_0 \ge C_0 \,t^d > 0, \quad t \in (0,\rho).
\end{equation}
Furthermore, for $\varepsilon>0$ there exists $r_{\alpha,t} \in \N$ such that for all $r \ge r_{\alpha,t}$
\begin{equation} \label{aux401}
 \left|\prod_{j=1}^{r} a_j(2^{-j}\alpha+2^{-j}t) - \widehat{\phi}(\alpha+t)\right|<\varepsilon,
\end{equation}
thus, this product decays slower that $t^d$ uniformly for all $r \ge r_{\alpha,t}$.
Making use of the Taylor expansion of $\hat{\phi}$ at $\alpha+2^r+t$,
$r \ge \max\{r_{0}, r_{\alpha,t}\}$, we obtain
\begin{equation} \label{aux402}
 |\hat{\phi}(\alpha+2^r+t)| \le C_2 \, q^{2^r} t^d, \quad q \in (0,1)
\end{equation}
with the constant $C_2>0$ depending on the constant $C>0$ that governs the exponential decay in \eqref{eq.200decay} and on the error term  in the Taylor expansion.
For $\ell=2^r$, $r \ge \max\{r_{0}, r_{\alpha,t}\}$, due to the 1-periodicity of the trigonometric
polynomials $a_j$, we have
$$
 \widehat{\phi}(\alpha+2^r+t)=\prod_{j=1}^r a_j(2^{-j}\alpha+2^{-j}t) \,
 \prod_{j=r+1}^{r+R} a_j(2^{-j}\alpha+2^{-j+r} +2^{-j}t) \, \prod_{j=r+R+1}^\infty
 a_j(2^{-j}\alpha+2^{-j+r}+2^{-j}t).
$$
Due to \eqref{aux400}, \eqref{aux401} and similar argument to \eqref{aux102}, the decay in \eqref{aux402} implies that
\begin{equation}\label{aux102_aa}
 \left|\prod_{j=r+1}^{r+R} a_j(2^{-j}\alpha+2^{-j+r}+2^{-j}t) \right| \le C_2 \, q^{2^r}, \quad q \in (0,1).
\end{equation}
Hence, at least one of the factors in \eqref{aux102_a} (\emph{has an almost zero}) is in the absolute value smaller than $C_2\, q^{2^r/R}$.
Repeating the argument with $\ell=2^{r+n}$, $n \in \N$, we conclude that $J$ trigonometric polynomials
$a_{r+1}, \ldots a_{r+1+J}$, $r \ge \max\{r_0, r_\alpha, r_{\alpha,t}\}$, $J >>R$, have at least $2(J-R)$  almost zeros. The possible almost zeros for each $a_k$, $k \in \{r+1,\ldots,r+1+J\}$ are at the distinct complex points in \eqref{aux102_aa}
\begin{equation}\label{aux403}
 2^{-k}\alpha+2^{-1}+2^{-k}\,t, \quad  2^{-k}\alpha +2^{-2}+2^{-k}\,t,  \quad \ldots \quad  2^{-k}\alpha +2^{-R}+2^{-k}\,t.
\end{equation}

\vspace{0.1cm}
\noindent \emph{3.Step} We choose the natural numbers  $s_1<s_2< \ldots < s_d$
and real numbers $t_1=2^{-s_1}, t_2=2^{-s_2}, \ldots, t_d=2^{-s_d}$ such that $t_j$ satisfy \eqref{aux400} and generate
(together with almost zeros in \emph{1.Step}) in total $(d+1)(J-R)$ distinct almost zeros in \eqref{aux403} for the $J$trigonometric
polynomials $a_{r+1}, \ldots a_{r+1+J}$, $r \ge \max\{r_0, r_\alpha, r_{\alpha,t}\}$.

\vspace{0.1cm}
\noindent \emph{4.Step} For the other non-zero values in the sequences in \eqref{eq.200decay} at points
$\alpha_\ell=\alpha+\ell$,  $\ell \in L \subset \N$, $|L| \le N$, we repeat the argument in \emph{1.Step}-\emph{2.Step} with the corresponding $t_j$ in \emph{3.Step}. Hence, we conclude that $J$ trigonometric polynomials
$a_{r+1}, \ldots a_{r+1+J}$, $r \ge \displaystyle \max_{\ell \in L \cup \{0\}}\{r_0, r_{\alpha_\ell}, r_{\alpha_\ell,t}\}$, $J >>R$, have at least $(N+1)(J-R)$ almost zeros of the form in \eqref{aux103} and in \eqref{aux403}. Repeating the argument in \emph{2.Step} of Proposition~\ref{prop:d=0}, by the pigeonhole principle,
there exists  $a_k$, $k \in \{r+1,\ldots,r+1+J\}$ with $N+1$ distinct almost zeros of the form in \eqref{aux103}
and in \eqref{aux403}.

\vspace{0.1cm}
\noindent The claim follows by the argument similar to the one in \emph{3.Step} of Proposition~\ref{prop:d=0} with the minimal distance of $2^{-k-s_d}$ between the almost zeros in \eqref{aux103}
and in \eqref{aux403}.
\end{proof}

\noindent A consequence of Theorem \ref{th:main_refinable} states that the analytic subspaces of the shift-invariant space $V_\phi$, in the case all the
trigonometric polynomials $a_j=a$, $j \in \N$, are the same, consist only of polynomials.

\begin{Corollary} \label{cor:stationary}
Let $V_\phi=\operatorname{span}\{\phi(\cdot-\ell)\ : \ \ell \in Z \}$ be defined by
$ \displaystyle \widehat{\phi}=\prod_{j=1}^\infty \, a(2^{-j} \cdot)$
with the trigonometric polynomial $a$ satisfying
$\operatorname{deg}(a) \le N$, $a(0)=1$ and $\|a\|_\infty \le C < \infty$.
If, for some $\lambda \in \C$ and $d \in \N_0$, $d \le N$, the analytic function $e^{\lambda \, t}\displaystyle \sum_{k=0}^{d} p_k\, \omega_k$,
$\omega_{d} \ne 0$,  belongs to $V_\phi$, then $\lambda=0$ and $\omega_k(t) \equiv \hbox{constant}$, $k=0, \ldots,d$.
\end{Corollary}
\begin{proof}
Let $\alpha=-\frac{\lambda i}{2\pi}$.
By Theorem \ref{th:main_refinable}, there are only finitely many $\ell \in \Z$ such that $\widehat{\phi}(\alpha+\ell) \not =0$.

\noindent \emph{Case $\lambda\not =0$} is impossible.
There exists at least one $\ell \in \Z$ (w.l.g $\ell=0$) such that
$\widehat{\phi}(\alpha+\ell) \not =0$. Otherwise, if $\widehat{\phi}(\alpha+\ell)=0$, $\ell \in \Z$, then \eqref{eq.200pis} implies that
$\omega_0(t)\equiv0$ and, by \eqref{eq.200basis}, the integer shifts of $\phi$ are linearly dependent.
Choose $R \in \N$ such that for all $r \ge R$ we have $\widehat{\phi}(\alpha+2^r) =0$. Ensuring these conditions we arrive at the
contradiction to the fact that the trigonometric polynomial $a$ has degree $N$. Indeed,  by the $1$-periodicity of $a$, we have
\begin{equation} \label{aux300}
 \widehat{\phi}(\alpha+2^r)=\underbrace{\prod_{j=1}^r a(2^{-j}\alpha)}_{\not=0} \prod_{j=r+1}^\infty
 a(2^{-j}\alpha+2^{-j+r})=0,
\end{equation}
which implies that at least one of the factors $ a(2^{-j}\alpha+2^{-j+r})=0$ for some $j \in \N$, $j \ge r+1$. None of such
factors, however, occur again for different $r$. Thus, to ensure that $\widehat{\phi}(\alpha+2^r) =0$, $r \in \N$, we are
forced to choose $a$ with infinitely many different zeros.

\noindent \emph{Case $\lambda=0$}: The claim that $\omega_k(t) \equiv \hbox{constant}$ for $k=0, \ldots,d$
is equivalent to $H_\lambda=\operatorname{span}\{1,\ldots,t^{d}\}$. By the assumption and by Theorem~\ref{th:structure_H_n1} $(ii)$, the subspace $H_\lambda$ is $d+1$ dimensional. Induction on $d$. For $d=0$, $\phi \in L_p(\R)$ and the
linear independence of its integer shifts imply that the zero condition $a(2^{-1})=0$ of order one is satisfied. This, together with
the standard normalization condition $a(0)=1$, imply that
$ \displaystyle\sum_{\ell \in \Z} \phi(t -\ell)=1$.  Next, we assume that $H_\lambda=\operatorname{span}\{1,\ldots,t^{d-1}\}$.
Therefore, by \cite{CDM},
$$
 D^k a(2^{-1})=0, \quad k=0,\ldots,d-1,
$$
and, by the $1$-periodicitity of $a$, we have, for $\ell=2\ell'+1$, $\ell' \in \Z\setminus \{0\}$,
$$
 D^k a(2^{-1}(2\ell'+1))=D^k a(2^{-1})=0, \quad k=0,\ldots,d-1.
$$
Hence, the expression for the $d$-th derivative of $\widehat{\phi}$ at such $\ell$ reduces to
\begin{eqnarray} \label{aux201}
\widehat{\phi}^{(d)}(\ell)&=&2^{-d} D^d a(2^{-1}) \prod_{j=2}^\infty a(2^{-j+1}\ell'+2^{-j})=
2^{-d} D^d a(2^{-1}) \prod_{j=1}^\infty a(2^{-j}\ell'+2^{-j}\cdot 2^{-1}) \notag  \\ &=&
2^{-d} D^d a(2^{-1}) \widehat{\phi}(\ell'+2^{-1}).
\end{eqnarray}
By Theorem \ref{th:main_refinable}, there exists infinitely many $\ell$ in \eqref{aux201} such that  $\widehat{\phi}^{(d)}(\ell)=0$.
Assuming that $\widehat{\phi}(\ell'+2^{-1})=0$ for all $\ell' \in \Z$ would contradict the linear independence of the
integer shifts of $\phi$, see \cite[Chapter 3.4]{Protasov_book}. Indeed, it would require that $a$ has zeros at complex
roots of $-1$, which implies polynomial structure of $H_\lambda$, but contradicts the linear independence.
Therefore,  $D^d a(2^{-1})=0$, which together with the inductive assumption implies that $H_\lambda=\operatorname{span}\{1, \ldots, t^d\}$.
\end{proof}

\begin{Remark}
Note that the case $\lambda \not=0$ is possible in the setting of generalized refinability as the infinitely
many zeros in \eqref{aux300} can be redistributed among the trigonometric polynomials $a_j(y)$, $j \in \N$.
\end{Remark}

\section{Generation properties of level dependent subdivision} \label{sec:subdivision}

\noindent In this section, we discuss the analytic limits of level dependent
subdivision schemes. In subsection \ref{subsec:subdivision_analytic},
we link the results from section \ref{sec:analytic_ref_1} with generation properties of
such subdivision schemes which are iterative algorithms
\begin{equation} \label{def:sub_scheme}
 c_{j+1}= S_{{\rm a}_{j}} c_{j}, \quad j \in \N,
\end{equation}
mapping sequences $c_{j}=\{ c_{j,k} \ : \ k \in \Z\}$ from $\ell(\Z)$ into $\ell(\Z)$.  The linear subdivision operators $ S_{{\rm a}_{j}}$ depend on the
finite sequences (\emph{masks}) ${\rm a}_{j}=\{ {\rm a}_{j,k} \ : \ k \in \Z\}$ of real numbers and are defined by
$$
 \left(S_{{\rm a}_{j}} c_{j} \right)_k=\sum_{m \in \Z} {\rm a}_{j,k-2m} c_{j,m}, \qquad k \in\Z,
  \qquad j \in \N.
$$
If the level dependent scheme associated with the sequence $\{{\rm a}_{j} \ : \ j \in \N\}$ of masks  converges, then all its
limits belong to $V_\phi$ with
$$
 \phi=\phi_1=\lim_{j \rightarrow \infty} S_{{\rm a}_{j}} \dots S_{{\rm a}_{1}} \delta, \qquad \delta_k=\left\{\begin{array}{cc} 1, & k=0 \\
0, & \hbox{otherwise}.  \end{array}\right.
$$

\subsection{Analytic limits} \label{subsec:subdivision_analytic}

\noindent We say that a level dependent scheme associated with the mask sequence $\{{\rm a}_{j} \ : \ j \in \N\}$ generates $U$,
if the subdivision limit $\displaystyle \lim_{j \rightarrow \infty} S_{{\rm a}_{j}} \dots S_{{\rm a}_{1}} c_1$ belongs to $U$ for
every starting sequence $c_1$ in \eqref{def:sub_scheme} sampled from a function in $U$.

\noindent The \emph{generation} properties of subdivision schemes  are well understood and are
characterized in terms of so-called \emph{zero conditions} or \emph{generalized zero conditions}, see e.g. \cite{CDM,DLL03,JePlo},
on the mask symbols
$$
 a^{[j]}(z)=\sum_{k \in \Z} {\rm a}_{j,k} \, z^\alpha, \quad j \in \N \quad z \in \C \setminus\{0\}.
$$
The zero conditions determine uniquely if the subdivision limit belongs to the exponential function space $U$ in Definition~\ref{def:U_exponential} or not.

\smallskip
\noindent Note that the requirement in \eqref{eq.200decay2} boils
down to the generalized zero conditions (or, equivalently, to the generalized Strang-Fix conditions \cite{Jia98, VonBluUnser07}) on the trigonometric
polynomials $a_j$ (or subdivision symbols $a^{[j]}$). We first illustrate this fact on the following example.

\begin{Example} \label{ex:relation_to_zero_condiitons}
It is well known \cite{DLL03} that the generation of two exponential polynomials $e^{\lambda t}$ and $t \, e^{\lambda t}$, $\lambda \in \C$, by a level dependent subdivision scheme is equivalent to the requirement that the
corresponding symbols
satisfy the generalized zero conditions (for $\lambda=0$, zero conditions at $-1$)
\begin{equation} \label{eq:zero_conditions_example}
 D^k \, a^{[j]}(-e^{-\lambda 2^{-j}})=0, \quad j\in \N, \quad k=0,1.
\end{equation}
In this case, generalized refinability \eqref{eq:refinable_equation} together with a standard assumptions
(for $\lambda=0$, conditions at $1$)
\begin{equation}\label{eq:normalization_at_one}
    D^k \, a^{[j]} (e^{-\lambda 2^{-j}}) \not =0,  \quad j\in \N, \quad k=0,1,
\end{equation}
imply that the only non-zero elements of the sequences in \eqref{eq.200decay}  for $k=0,1$ are
$$
 \widehat{\phi}\left(-\frac{i \lambda}{2 \pi} \right) \not=0 \quad \hbox{and} \quad
 \widehat{\phi}^{(1)}\left(-\frac{i \lambda}{2 \pi} \right) \not=0.
$$
In other words, by Theorem \ref{th:structure_H_n1}, the $1$-periodic analytic functions $\omega_0$ and $\omega_1$
are constant.
Indeed, a straightforward computation yields
\begin{equation}\label{eq:aux2}
 \widehat{\phi}\left(-\frac{i \lambda}{2 \pi} \right)=
 \prod_{j=1}^\infty \, a_j\left(-2^{-j}\, \frac{i \lambda}{2 \pi} \right)=
 \prod_{j=1}^\infty \, a^{[j]}\left(e^{-\lambda \, 2^{-j}} \right) \not=0
\end{equation}
and
\begin{equation}\label{eq:aux3}
\widehat{\phi}^{(1)}\left(-\frac{i \lambda}{2 \pi} \right)= \sum_{\ell \in \N}
2^{-\ell} D a_\ell^{*}\left(e^{- \lambda 2^{-j}} \right) \prod_{j=1 \atop j \not=\ell}^\infty \, a^{[j]}\left(e^{- \lambda 2^{-j}} \right) \not=0 \,.
\end{equation}
Furthermore, for $\beta \in \Z \setminus \{0\}$, let $j' \in \N$ be the number
of $2$'s in the prime number decomposition of $|\beta|$. Define $j=j'+1$. Then we have
$$
   D^k a_j \left(-\frac{i \lambda}{2 \pi} +\beta \right)= D^k a^{[j]}\left(-e^{- \lambda \, 2^{-j}} \right), \quad k=0,1.
$$
Therefore, by  \eqref{eq:zero_conditions_example},  we arrive at the generalized Strang-Fix conditions
\begin{equation}\label{eq:aux4}
D^{(k)}\widehat{\phi}\left(-\frac{i \lambda}{2 \pi} +\beta\right)=0
\quad \hbox{for $k=0,1$ and $\beta \in \Z \setminus\{0\}$}.
\end{equation}

\noindent Similarly, if \eqref{eq:aux2}-\eqref{eq:aux4} are satisfied, then, by Theorem \ref{th:structure_H_n1},
$H_\lambda$ is spanned by $e^{\lambda t}$ and $t \, e^{\lambda t}$, which is equivalent to
\eqref{eq:zero_conditions_example}.
\end{Example}

\noindent The main result of subsection \ref{sec:analytic_ref_1},
Theorem \ref{th:main_refinable}, essentially states
that if a function belongs to the subspace $H \subseteq V_\phi$ of analytic functions, then
it must satisfy the generalized zero conditions. Thus, completing the quest for exhibiting
all possible analytic functions generated by level dependent subdivision. We restate
Theorem \ref{th:main_refinable}.

\begin{Theorem}
 Every analytic limit of a level dependent subdivision scheme is an exponential polynomial.
\end{Theorem}

\noindent In the special level independent (stationary) case, i.e. ${\rm a}_{j}={\rm a}$ for all $j \in N$. The Corollary
\ref{cor:stationary} can be restated as follows.

\begin{Corollary}
 Every analytic limit of a level independent subdivision scheme is a polynomial.
\end{Corollary}

\smallskip
\noindent \textbf{Acknowledgement:} {\bf We are deeply indebted to Nira Dyn for profound discussions on the topic of this paper.}
Maria Charina was sponsored by the Austrian Science Foundation (FWF) grant P28287-N35.
Vladimir Protasov is supported by RFBR grants No 17-01-00809 and 19-04-01227.
Both authors are grateful to the Erwin Schroedinger Institute, Vienna, Austria, for the productive discussion atmosphere.


\end{document}